\documentclass{amsart} 
\usepackage{amsmath,amssymb,amsthm,amsfonts,amsxtra,mathtools}
\usepackage{enumerate,verbatim,mathrsfs,comment,color,url}
\usepackage[all,2cell,ps]{xy}
\usepackage[pagebackref]{hyperref}
\usepackage{graphicx}
\usepackage{verbatim}

\DeclareMathOperator{\cid}{CI-dim}

\DeclareMathOperator{\curv}{curv}
\DeclareMathOperator{\cx}{cx}
\DeclareMathOperator{\depth}{depth}

\DeclareMathOperator{\Ext}{Ext}

\DeclareMathOperator{\gdim}{G-dim}

\DeclareMathOperator{\injdim}{injdim}

\DeclareMathOperator{\pd}{pd}

\DeclareMathOperator{\rank}{rank}

\DeclareMathOperator{\Tor}{Tor}

\renewcommand{\ge}{\geqslant}
\renewcommand{\le}{\leqslant}

\newcommand{\fm}{\mathfrak{m}}
\newcommand{\fn}{\mathfrak{n}}

\renewcommand{\iff}{if and only if }

\theoremstyle{plain}
\newtheorem{theorem}{Theorem}[section]
\newtheorem{lemma}[theorem]{Lemma}

\newtheorem{corollary}[theorem]{Corollary}

\theoremstyle{definition}
\newtheorem{definition}[theorem]{Definition}

\newtheorem{example}[theorem]{Example}

\newtheorem{para}[theorem]{}

\newtheorem{setup}[theorem]{Setup}

\theoremstyle{remark}
\newtheorem{remark}[theorem]{Remark}

\numberwithin{equation}{section}

\title[Integrally closed $\fm$-primary ideals have extremal resolutions]{Integrally closed $\fm$-primary ideals have extremal resolutions}

\author[D.~Ghosh]{Dipankar Ghosh}
\address{Department of Mathematics, Indian Institute of Technology Kharagpur, West Bengal - 721302, India}
\email{dipankar@maths.iitkgp.ac.in, dipug23@gmail.com}

\author[T.J.~Puthenpurakal]{Tony J. Puthenpurakal}
\address{Department of Mathematics, Indian Institute of Technology Bombay, Powai, Mumbai 400076, India}
\email{tputhen@math.iitb.ac.in}
\date{August 30, 2022}

\begin{document}

\pagenumbering{arabic}
\thispagestyle{empty}
  
 \begin{abstract}
 	 We show that every integrally closed $\fm$-primary ideal $I$ in a commutative Noetherian local ring $(R,\fm,k)$ has maximal complexity and curvature, i.e., $ \cx_R(I) = \cx_R(k) $ and $ \curv_R(I) = \curv_R(k) $. As a consequence, we characterize complete intersection local rings in terms of complexity, curvature and complete intersection dimension of such ideals. The analogous results on projective, injective and Gorenstein dimensions are known. However, we provide short proofs of these results as well.
 \end{abstract}

\maketitle

\section{Introduction}

The notions of complexity and curvature were introduced by Avramov. They measure the growth of Betti numbers, and distinguish the modules of infinite projective dimension. The aim of this article is to study complexity and curvature of certain integrally closed ideals, and characterize complete intersection local rings in terms of complexity, curvature and complete intersection dimension of such ideals.

\begin{setup}\label{setup}
	Throughout, $(R,\fm,k)$ is a commutative Noetherian local ring. Let $I$ be an integrally closed $\fm$-primary ideal of $R$.
\end{setup}

It is shown in \cite[Cor.~3.3]{CHKV06} by Corso-Huneke-Katz-Vasconcelos that $I$ is rigid-test as an $R$-module, i.e., for every finitely generated $R$-module $N$, $\Tor_n^R(I,N)=0$ for some $ n \ge 1 $ implies that $ \Tor_{i \ge n}^R(I,N) = 0 $ and $\pd_R(N) \le n$. Over a complete intersection local ring, Celikbas-Dao-Takahashi \cite[Thm.~1.4]{CDT14} showed that the test modules are precisely the modules of maximal complexity. It follows from the above two results that $I$ has maximal complexity when $R$ is a complete intersection local ring. In the present article, we prove this result over an arbitrary Noetherian local ring. More generally, in Theorem~\ref{thm:int-closed-ideal-max-cx-curv}, we show that $I$ has maximal complexity as well as curvature, i.e., $ \cx_R(I) = \cx_R(k) $ and $ \curv_R(I) = \curv_R(k) $. See \ref{def:complexity} and \ref{para:max-cx-curv} for the definitions of (maximal) complexity and curvature.

Theorem~\ref{thm:int-closed-ideal-max-cx-curv} provides a new class of ideals, namely integrally closed $\fm$-primary ideals, having maximal complexity and curvature. In this context, it is known that every nonzero homomorphic image of a finite direct sum of syzygy modules of the residue field $k$ has maximal complexity and curvature, due to Avramov \cite[Cor.~9]{Avr96}. Every Ulrich module over a Cohen-Macaulay local ring also has maximal complexity and curvature, see \cite[Thm.~3.7]{DG22}.

The main motivation for Theorem~\ref{thm:int-closed-ideal-max-cx-curv} came from the following results. Burch \cite[pp~947, Cor.~3]{Bu68} proved that $R$ is regular \iff projective dimension $\pd_R(I)$ is finite. We recover this result as a corollary of Theorem~\ref{thm:int-closed-ideal-max-cx-curv}, see Corollary~\ref{cor:Burch}. The counterpart of Burch's result for Gorenstein dimension is shown by Celikbas--Sather-Wagstaff in \cite[Thm.~1.1]{CS16}. The analogous result for injective dimension is also known by \cite[6.12]{CGSZ18}. As a consequence of Theorem~\ref{thm:int-closed-ideal-max-cx-curv}, in Corollary~\ref{cor:int-closed-ideal-max-cx-curv}, not only we obtain the counterpart of all these results for complete intersection dimension, but also we strengthen this to complexity and curvature of such ideals. We show that the following are equivalent: (1) $R$ is complete intersection, (2) $\cid_R(I) < \infty$ (where CI-dim denotes complete intersection dimension), (3) $\cx_R(I) < \infty$, and (4) $\curv_R(I) \le 1$. Note that for a finitely generated $R$-module $M$, the implications $(\ast)$ and $(\ast\ast)$ in the diagram below hold true, see \cite[(5.3)]{AGP97} and \cite[4.2.3.(4)]{Avr98}.
\begin{displaymath}
\xymatrix{
\cid_R(M) < \infty \ar@/{}^{1.3pc}/@{=>}[r]^{(\ast)} & \cx_R(M) < \infty \ar@/{}^{1.3pc}/@{=>}[l]_{(\dagger)} \ar@/{}^{1.3pc}/@{=>}[r]^{(\ast\ast)} & \curv_R(M) \le 1 \ar@/{}^{1.3pc}/@{=>}[l]_{(\dagger\dagger)} }
\end{displaymath}
The implication $(\dagger)$ is not true in general. In fact, there are examples \cite[(3.3)]{GP90} of modules over Artinian local rings having periodic resolution with minimal period $q > 2$. Such modules cannot have finite CI-dimension, cf.~\cite[Thm.~7.3.(1)]{AGP97}. As far as we know, $(\dagger\dagger)$ is open in general, cf.~\cite[Prob.~4.3.7]{Avr98}. However, both $(\dagger)$ and $(\dagger\dagger)$ hold true for modules with maximal complexity and curvature, see, e.g., \cite[Thm.~3]{Avr96} and \cite[(1.3)]{AGP97}.

Finally, we give different and short proofs of the counterparts of Burch's result for injective and Gorenstein dimensions, see Theorems~\ref{thm:int-closed-ideal-finite-injdim} and \ref{thm:G-dim-of-I}.

\section{Results on complexity and curvature, and their consequences}

Throughout, we work with Setup~\ref{setup}. All $R$-modules are assumed to be finitely generated. For an $R$-module $M$, let $ \beta_n^R(M) $ denote the $n$th Betti number of $M$, i.e., $ \beta_n^R(M) = \rank_k\left( \Ext_R^n(M,k) \right) $.

\begin{definition}\label{def:complexity}
	(1) The complexity of $ M $, denoted $ \cx_R(M) $, is the smallest non-negative integer $ b $ such that there exists a real number $ \alpha > 0 $ with $ \beta_n^R(M) \le \alpha n^{b-1} $ for all $ n \gg 0 $. If no such $ b $ exists, then $ \cx_R(M) := \infty $.
	
	(2) The curvature of $ M $, denoted $ \curv_R(M) $, is the reciprocal value of the radius of convergence of the Poincar\'{e} series $ P_M^R(t) := \sum_{n \ge 0} \beta_n^R(M) t^n $ of $ M $, i.e.,
	\[
		\curv_R(M) := \limsup_{n \to \infty} \sqrt[n]{\beta_n^R(M)}.
	\]
	Note that $ \curv_R(M) $ is always finite (cf.~\cite[4.2.3.(5)]{Avr98}).
\end{definition}

We refer the reader to \cite[Section~4.2]{Avr98} for the basic properties of complexity and curvature. In our study, we particularly use the following results.

\begin{para}\label{para:max-cx-curv}
	It is known that the residue field has maximal complexity and curvature, i.e.,
	\begin{align*}
	\cx_R(k) & = \sup\{ \cx_R(M) : M \mbox{ is an $ R $-module} \} \,\mbox{ and}\\
	\curv_R(k) & = \sup\{ \curv_R(M) : M \mbox{ is an $ R $-module} \},
	\end{align*}
	see \cite[4.2.4.(1)]{Avr98}.
	So an $ R $-module $ M $ is said to have maximal complexity (resp., curvature) if $ \cx_R(M) = \cx_R(k) $ (resp., $ \curv_R(M) = \curv_R(k) $).
\end{para}

\begin{para}\label{para:cx-curv-direct-sum}
	For $R$-modules $M$ and $N$, since $\beta_n^R(M \oplus N) = \beta_n^R(M) + \beta_n^R(N)$ for every $n \ge 0$, the definitions of complexity and curvature yield that
	\begin{align*}
	\cx_R(M \oplus N) &= \max\{ \cx_R(M), \cx_R(N) \} \quad \mbox{and}\\
	\curv_R(M \oplus N) &= \max\{ \curv_R(M), \curv_R(N) \} \quad \mbox{(cf.~\cite[4.2.4.(3)]{Avr98})}.
	\end{align*}
\end{para}

In order to prove our main results, when $\depth(R) = 0$, we replace $R$ by $R[X]_{\langle \fm, X\rangle}$, and assume that $\depth(R) \ge 1$.

\begin{lemma}\label{lem:S-R[[X]]}
	Set $S:=R[X]_{\langle \fm, X\rangle}$. Let $J$ be an ideal of $R$. Suppose $\fn$ is the maximal ideal of $S$. Then $X \in \fn\smallsetminus\fn^2$ is $S$-regular, and $S/XS \cong R$. Moreover,
	\begin{enumerate}[{\rm (1)}]
		\item $(JS+XS)/X(JS+XS) \cong J \oplus R/J$ as modules over $R$.
		\item $\cx_S(JS+XS) = \cx_R(J)$ and $\curv_S(JS+XS) = \curv_R(J)$.
		\item The ideal $J$ of $R$ is of maximal complexity $($resp., curvature$)$ \iff the ideal $JS+XS$ of $S$ is so.
		\item The ideal $J$ is integrally closed $($resp., $\fm$-primary$)$ in $R$ \iff $JS+XS$ is integrally closed $($resp., $\fn$-primary$)$ in $S$.
	\end{enumerate}
\end{lemma}

\begin{proof}
	(1) As $R$-modules, we have the following isomorphisms:
	\begin{align*}
	\dfrac{JS+XS}{X(JS+XS)} &\cong T^{-1} \left(\dfrac{J R[X] + X R[X]}{XJR[X] + X^2 R[X]}\right) \quad \mbox{where $T:=R[X]\smallsetminus\langle \fm, X\rangle$}\\
	& \cong T^{-1} \left(\dfrac{J \oplus XR \oplus X^2 R[X]}{0\oplus XJ \oplus X^2 R[X]}\right) \cong J \oplus \frac{R}{J}.
	\end{align*}
	
	(2) For an $S$-module $N$, if $z$ is a regular element over both $S$ and $N$, then $\beta_n^S(N) = \beta_n^{S/zS}(N/zN)$ for every $n\ge 0$, see, e.g., \cite[p.~140, Lem.~2]{Mat86}. Therefore, since $X$ is regular over both $S$ and $JS+XS$, it follows that
	\begin{align*}
	\cx_S(JS+XS) &= \cx_R\big((JS+XS)/X(JS+XS)\big)\\
	&= \cx_R(J \oplus R/J) \quad \mbox{[by (1)]}\\
	&= \max\{ \cx_R(J), \cx_R(R/J) \} \quad \mbox{[by \ref{para:cx-curv-direct-sum}]}\\
	&= \cx_R(J) \quad \mbox{[as $\cx_R(J) = \cx_R(R/J)$]}.
	\end{align*}
	The equality $\curv_S(JS+XS) = \curv_R(J)$ can be obtained similarly.
	
	(3) Since $X \in \fn\smallsetminus\fn^2$ is $S$-regular, it is well known that
	\begin{center}
		$\cx_S(k) = \cx_{R}(k)$ and $\curv_S(k) = \curv_{R}(k)$,
	\end{center}
	see, e.g., \cite[3.3.5]{Avr98}.
	So one obtains the desired assertions from (2).
	
	(4) When $J$ is integrally closed in $R$, by \cite[Prop.~1.3.5]{SH06}, the ideal $\langle J, X \rangle$ of $R[X]$ is integrally closed. Since integral closure commutes with localization (cf.~\cite[1.1.4]{SH06}), the ideal $JS+XS$ is integrally closed in $S$. For the converse, assume that $JS+XS$ is integrally closed in $S$. Let $r \in R$ be integral over $J$. The equation of integral dependence of $r$ over $J$ yields that $\frac{r}{1} \in \overline{JS+XS} = JS+XS$. It follows that there exists a unit $u$ of $R$ such that $ur \in J$, hence $r\in J$. So $J$ is also integrally closed in $R$. The statement on primary ideals follows from the fact that $\fm^n \subseteq J$ \iff $\fn^n = \fm^nS + XS \subseteq JS+XS $, where $n \ge 1$ is an integer.
\end{proof}

We use the following remark in Section~\ref{sec:injdim-G-dim}.

\begin{remark}\label{rmk:Lemma-G-dim}
	In Lemma~\ref{lem:S-R[[X]]}, $R$ is Gorenstein \iff $S$ is so. Moreover, in view of the results \cite[1.4.5, 1.2.7 and 1.2.9]{Cr00} on Gorenstein dimension, one has that
	\begin{align*}
	\gdim_S(JS+XS) &= \gdim_R\big((JS+XS)/X(JS+XS)\big)\\
	&=\gdim_R\Big(J\oplus \frac{R}{J}\Big) = \max\{ \gdim_R(J), \gdim_R(R/J) \}.
	\end{align*}
\end{remark}

The main ingredients of the proof of our main result (Theorem~\ref{thm:int-closed-ideal-max-cx-curv}) is the use of $\fm$-full ideals and their properties.

\begin{para}\label{para:m-full-ideals-properties}
	Let $J$ be an ideal of $R$.
	\begin{enumerate}[\rm (1)]
		\item The ideal $J$ is called $\fm$-full if $(\fm J :_R x) = J$ for some $x\in \fm$.
		\item Suppose that $J$ is integrally closed, and the residue field $k$ is infinite. Then, by \cite[Thm.~2.4]{Go87}, either $J$ is $\fm$-full or $J=\sqrt{0}$ (the nilradical of $R$). Moreover, if $\depth(R)>0$ and $J$ is $\fm$-full, from the proof of \cite[Thm.~2.4]{Go87}, the element $x$ can be chosen as a nonzero-divisor on $R$ such that $x \in \fm \smallsetminus \fm^2$ and $(\fm J :_R x) = J$.
		\item If $J=\sqrt{0}$, then $J^m = 0$ for some $m$. On the other hand, if $J$ is $\fm$-primary, then $\fm^l\subseteq J$ for some $l$. Thus, if $J=\sqrt{0}$ and $J$ is $\fm$-primary, then $\fm^n=0$ for some $n$, which implies that $\dim(R)=0$.
		\item \label{k-summand-of-I-mod-x}Let $\depth(R)>0$, and $k$ is infinite. Let $J$ be an $\fm$-primary integrally closed ideal of $R$. Then it follows from (2) and (3) above that there exists an $R$-regular element $x\in \fm \smallsetminus \fm^2$ such that $(\fm J :_R x) = J$. Hence, by \cite[Prop.~2.3.(5)]{AT05}, $k$ is a direct summand of $J/xJ$ as an $R/xR$-module.
	\end{enumerate}
\end{para}

Now we are in a position to prove our main result.

\begin{theorem}\label{thm:int-closed-ideal-max-cx-curv}
	With {\rm Setup~\ref{setup}}, the ideal $I$ has maximal complexity and curvature.
\end{theorem}

\begin{proof}
	In view of Lemma~\ref{lem:S-R[[X]]}, without loss of generality, we may assume that $\depth(R) > 0$. If the residue field is finite, then passing through $R[X]_{\fm R[X]}$, we may also assume that $R$ has infinite residue field $k$. As $I$ is integrally closed and $\fm$-primary, by \ref{para:m-full-ideals-properties}.\eqref{k-summand-of-I-mod-x}, there exists an $R$-regular element $x\in \fm \smallsetminus \fm^2$ such that
	\begin{equation}\label{eqn:k-direct-summand-of-I-mod-x}
		\mbox{$k$ is a direct summand of $I/xI$ as an $R/xR$-module.}
	\end{equation}
	Since $x$ is $R$-regular, for all $ n\ge 0 $, the $n$th Betti number of $I/xI$ as an $R/xR$-module is same as that of $I$ as an $R$-module. It follows that
	\begin{equation}\label{equalities-cx-curv-mod-x}
	\cx_{R/xR}(I/xI)=\cx_R(I) \quad \mbox{and} \quad \curv_{R/xR}(I/xI)=\curv_R(I).
	\end{equation}
	Moreover, the following (in)equalities hold true:
	\begin{align}\label{inequality-cx-curv}
	\cx_R(k) & = \cx_{R/xR}(k) \quad \mbox{[see, e.g., \cite[Lem.~3.3]{DG22}]} \\
	& \le \cx_{R/xR}(I/xI) \quad \mbox{[by \eqref{eqn:k-direct-summand-of-I-mod-x} and Section~\ref{para:cx-curv-direct-sum}]} \nonumber\\
	& =\cx_R(I) \quad \mbox{[by \eqref{equalities-cx-curv-mod-x}]} \nonumber\\
	& \le \cx_R(k)\quad \mbox{[by Section~\ref{para:max-cx-curv}]}.\nonumber
	\end{align}
	Hence one concludes that $ \cx_R(I) = \cx_R(k) $. The similar (in)equalities as in \eqref{inequality-cx-curv} hold true for curvature, which yield that $ \curv_R(I) = \curv_R(k) $.
\end{proof}

The main application of Theorem~\ref{thm:int-closed-ideal-max-cx-curv} is the following.

\begin{corollary}\label{cor:int-closed-ideal-max-cx-curv}
	With {\rm Setup~\ref{setup}}, the following are equivalent:
	\begin{enumerate}[{\rm (1)}]
		\item $R$ is complete intersection.
		\item $\cid_R(I) < \infty$ $($equivalently, $\cid_R(R/I) < \infty$$)$.
		\item $\cx_R(I) < \infty$.
		\item $\curv_R(I) \le 1$.
	\end{enumerate}
\end{corollary}

\begin{proof}
	(1) $\Rightarrow$ (2): It follows from \cite[(1.3)]{AGP97}.
	
	(2) $\Rightarrow$ (3): If $ \cid_R(I) < \infty $, then $ \cx_R(I) < \infty $, cf.~\cite[(5.3)]{AGP97}.
	
	(3) $\Rightarrow$ (4) $\Rightarrow$ (1): One concludes these implications from Theorem~\ref{thm:int-closed-ideal-max-cx-curv} and the well known facts that $R$ is complete intersection \iff $\cx_R(k) < \infty$ \iff $\curv_R(k) \le 1$, see, e.g., \cite[8.1.2 and 8.2.2]{Avr98}.
\end{proof}

\begin{remark}\label{rmk:CI-dim-I-R-CI}
	Tavanfar, in \cite[Thm.~3.4.4]{Ta19}, showed that a test $R$-module has finite CI-dimension \iff $R$ is complete intersection. Thus the equivalence (1) $\Leftrightarrow$ (2) in Corollary~\ref{cor:int-closed-ideal-max-cx-curv} also follows from \cite[Cor.~3.3]{CHKV06} and \cite[Thm.~3.4.4]{Ta19}. We should note that the equivalences (1) $\Leftrightarrow$ (3) and (1) $\Leftrightarrow$ (4) considerably strengthen the equivalence (1) $\Leftrightarrow$ (2) in Corollary~\ref{cor:int-closed-ideal-max-cx-curv}.
\end{remark}

As another consequence of Theorem~\ref{thm:int-closed-ideal-max-cx-curv}, one recovers the result of Burch.

\begin{corollary}{\rm (Burch \cite[pp~947, Cor.~3]{Bu68})}\label{cor:Burch}
	With {\rm Setup~\ref{setup}},
	\begin{center}
		$R$ is regular $\Longleftrightarrow$ $ \pd_R(I) < \infty $ $($equivalently, $\pd_R(R/I) < \infty$$)$.
	\end{center}
\end{corollary}

\begin{proof}
	For an $R$-module $M$, one has that $\pd_R(M)$ is finite \iff $\cx_R(M) =0 $. Therefore, if $ \pd_R(I) $ is finite, then by Theorem~\ref{thm:int-closed-ideal-max-cx-curv}, $\cx_R(k) = \cx_R(I) = 0$, hence $\pd_R(k)$ is finite, which implies that $R$ is regular.
\end{proof}

We collect some examples which complement Theorem~\ref{thm:int-closed-ideal-max-cx-curv} and Corollary~\ref{cor:int-closed-ideal-max-cx-curv}.

\begin{example}
	An integrally closed $\fm$-primary ideal need not be the same as $\fm$ (it is when $R$ is Artinian). For example, consider $R = k[[x,y]]/(x^2,xy)$ over a field $k$, and let $I$ be the integral closure of $(y^2)$ in $R$. Since $I$ does not contain $y$, the ideal $I$ is not the same as $\fm = \langle x,y \rangle$. In fact $I = \langle x,y^2 \rangle = \langle x \rangle \oplus \langle y^2 \rangle \cong k \oplus \langle y^2 \rangle $. So $I$ has maximal complexity and curvature.
\end{example}

Theorem~\ref{thm:int-closed-ideal-max-cx-curv} and Corollary~\ref{cor:int-closed-ideal-max-cx-curv} do not necessarily hold true for arbitrary $\fm$-primary ideals which are not integrally closed.

\begin{example}\label{exam-1:counter-corollaries}
	Let $R=k[x,y,z]/(x^2,xy,y^2,z^2)$ over a field $k$. Set $J:=(z)$. Then $J$ is an $\fm$-primary ideal, but not integrally closed. Note that $R$ is an Artinian local ring, but not a complete intersection ring, hence $\cx_R(k) = \infty$ and $\curv_R(k) \ge 2$. Since $\beta_n^R(J)=1$ for all $n\ge 0$, one obtains that $\cx_R(J) = 1 = \curv_R(J)$. Thus $J$ does not have maximal complexity and curvature. Moreover, we notice that $\cid_R(R/J) = 0$. Indeed, consider the quasi-deformation $R \stackrel{=}{\rightarrow} R \twoheadleftarrow Q$ of $R$, where $Q = k[X,Y,Z]/(X^2,XY,Y^2)$. Since $\pd_Q(R/J)-\pd_Q(R) = \pd_Q(Q/(z)) - \pd_Q(Q/(z^2)) = 1-1 = 0$, it follows that $\cid_R(R/J) = 0 < \infty$.
\end{example}

Theorem~\ref{thm:int-closed-ideal-max-cx-curv} and Corollary~\ref{cor:int-closed-ideal-max-cx-curv} do not necessarily hold true for arbitrary integrally closed ideals which are not $\fm$-primary.

\begin{example}\label{exam-2:counter-corollaries}
	Let $R=k[[x,y,z,w]]/(xy,yz,zx,w^2)$ over a field $k$. Set $J:=(w)$. Then $J$ is an integrally closed ideal of $R$, but not $\fm$-primary. Here $R$ is a CM local ring of dimension $1$, but not complete intersection. With the similar arguments as in Example~\ref{exam-1:counter-corollaries}, one obtains that $J$ does not have maximal complexity and curvature. Moreover, $\cid_R(R/J) = 0 < \infty$.
\end{example}

Next we observe that the condition `$\fm$-primary' is not a necessary condition for an integrally closed ideal to have maximal complexity and curvature.

\begin{example}
	Let $R=k[[x,y,z]]/(xy,yz,zx)$ over a field $k$. Set $J:=(x)$. Then $R$ is a CM local ring of dimension $1$, and the ideal $J$ is integrally closed, but not $\fm$-primary. Since the 1st syzygy module of $(x)$ is given by $\Omega^R_1((x)) = (y,z) = (y) \oplus (z)$, inductively, one derives that $\beta_n^R(J) = \mu(\Omega^R_n(J)) = 2^n$ for all $n\ge 0$. Therefore $\cx_R(J) = \infty$ and $\curv_R(J) = 2$. Since $\fm = (x,y,z) = (x) \oplus (y) \oplus (z) $, it follows that $\beta_n^R(\fm) = 3 \times 2^n$ for all $n\ge 0$. Hence $\cx_R(\fm) = \infty$ and $\curv_R(\fm) = 2$. Thus $J$ has maximal complexity and curvature, though $J$ is not an $\fm$-primary ideal.
\end{example}

\section{Analogous results on injective and Gorenstein dimensions}\label{sec:injdim-G-dim}

The equivalence (1) $\Leftrightarrow$ (3) in the theorem below is due to \cite[6.12]{CGSZ18}. However, after using two classical results due to Roberts and Peskine-Szpiro, the proof of the following theorem is easy and short compare to that of \cite[6.12]{CGSZ18}.

\begin{theorem}\label{thm:int-closed-ideal-finite-injdim}
	With {\rm Setup~\ref{setup}}, the following three conditions are equivalent:\\
	{\rm (1)} $R$ is regular, {\rm (2)} $\injdim_R(R/I) < \infty$, and {\rm (3)} $\injdim_R(I) < \infty$.
\end{theorem}

\begin{proof}
	The implications (1) $\Rightarrow$ (2) and (1) $\Rightarrow$ (3) are well known.
	
	(2) $\Rightarrow$ (3): Let $\injdim_R(R/I) < \infty$. By a result of Peskine and Szpiro \cite[Chapitre II, Th\'{e}or\`{e}me (5.5)]{PS73}, as there is a nonzero cyclic $R$-module of finite injective dimension, $R$ is Gorenstein, i.e., $\injdim_R(R) < \infty$. Then, from the exact sequence $0 \to I  \to R \to R/I \to 0$, it follows that $\injdim_R(I) < \infty$.
	
	(3) $\Rightarrow$ (1): Let $\injdim_R(I) < \infty$. Since $I$ is $\fm$-primary, $\fm^n \subseteq I$ for some $n \ge 1$. If $I=0$, then $\fm^n = 0$, hence $\fm \subseteq \overline{I} = I$, i.e., $\fm = I = 0$, which implies that $R$ is a field. So we may assume that $I \neq 0$. The existence of a nonzero (finitely generated) $R$-module of finite injective dimension ensures that $R$ is Cohen-Macaulay (see, e.g., \cite[9.6.2, 9.6.4(ii)]{BH93} and \cite{Rob87}). We claim that $R$ cannot be Artinian. Otherwise, if $R$ is Artinian, then $\fm^u =0$ for some $u \ge 1$. As above, one gets that $\fm = I \neq 0$, and $\injdim_R(\fm) < \infty$, which implies that $R$ is regular (cf.~\cite[Thm.~3.7]{GGP18}),	hence $I = \fm = 0$, a contradiction. So $\depth(R)=\dim(R) \ge 1$. Then, as in the proof of Theorem~\ref{thm:int-closed-ideal-max-cx-curv}, there is an $R$-regular element $x \in \fm \smallsetminus \fm^2$ such that $k$ is a direct summand of $I/xI$ as an $R/xR$-module. It follows that
	\[
		\injdim_{R/xR}(k) \le \injdim_{R/xR}(I/xI) = \injdim_R(I) - 1 < \infty,
	\]
	see, e.g., \cite[3.1.15]{BH93}. Therefore $R/xR$ is regular, and hence $R$ is regular.
\end{proof}

A similar method as in the proof of Theorem~\ref{thm:int-closed-ideal-finite-injdim} yields a simple and elementary proof of the analogous result on Gorenstein dimension.

\begin{theorem}\cite[Thm.~1.1]{CS16}\label{thm:G-dim-of-I}
	With {\rm Setup~\ref{setup}},
	\begin{center}
		$R$ is Gorenstein $\Longleftrightarrow$ $\gdim_R(I) < \infty$ $($equivalently, $\gdim_R(R/I) < \infty$$)$.
	\end{center}
\end{theorem}

\begin{proof}
	The implication $(\Rightarrow)$ is well known, cf.~\cite[1.4.9]{Cr00}. For the reverse implication, we may assume that $I\neq 0$ (as in the proof of Theorem~\ref{thm:int-closed-ideal-finite-injdim}), and $\depth(R)>0$ (by Lemma~\ref{lem:S-R[[X]]} and Remark~\ref{rmk:Lemma-G-dim}). Furthermore, we may assume that $k$ is infinite. Therefore, by \ref{para:m-full-ideals-properties}.\eqref{k-summand-of-I-mod-x}, there is an $R$-regular element $x \in \fm \smallsetminus \fm^2$ such that $k$ is a direct summand of $I/xI$ as an $R/xR$-module. Hence
	\begin{align*}
	\gdim_R(k) - 1 &= \gdim_{R/xR}(k) \quad \mbox{[by \cite[(1.5.2)]{Cr00}]}\\
	& \le \gdim_{R/xR}(I/xI) \quad \mbox{[by \cite[1.2.7 and 1.2.9]{Cr00}]}\\
	& = \gdim_R(I)  \quad \mbox{[by \cite[1.4.5]{Cr00}]}.
	\end{align*}
	Thus, if $\gdim_R(I) < \infty$, then $\gdim_R(k) < \infty$, hence $R$ is Gorenstein.
\end{proof}
%

\section*{Acknowledgments}
Ghosh was supported by Start-up Research Grant (SRG) from SERB, DST, Govt.~of India with the Grant No SRG/2020/000597.


\begin{thebibliography}{AAAA}
	
%
			
	\bibitem{AT05} J.~Asadollahi and T.J.~Puthenpurakal, {\it An analogue of a theorem due to Levin and Vasconcelos}. Commutative algebra and algebraic geometry, 9--15, Contemp. Math., 390, 
	Amer. Math. Soc., Providence, RI, 2005.
	
%
%
	
	\bibitem{Avr96} L. L. Avramov, {\it Modules with extremal resolutions}, Math. Res. Lett. {\bf 3} (1996), 319--328.
	
	\bibitem{Avr98} L. L. Avramov, {\it Infinite free resolutions}, Six lectures on commutative algebra, Bellaterra 1996, Progr. Math. 166, Birkh\"{a}user, Basel, (1998), 1--118.
	
	\bibitem{AGP97} L.~L.~Avramov, V.~N.~Gasharov and I.~V.~ Peeva, {\it Complete intersection dimension}, Inst. Hautes \'{E}tudes Sci. Publ. Math. {\bf 86} (1997), 67--114.
	
	
	\bibitem{BH93} W.~Bruns and J.~Herzog, \emph{Cohen-{M}acaulay rings}, Cambridge Studies in Advanced Mathematics, vol.~39, Cambridge University Press, Cambridge, 1993.
	
	\bibitem{Bu68} L.~Burch, {\it On ideals of finite homological dimension in local rings}, Proc. Camb. Philos. Soc. {\bf 64}, (1968), 941--948.
	
	\bibitem{CDT14} O.~Celikbas, H.~Dao and R.~Takahashi, Modules that detect finite homological dimensions. Kyoto J. Math. {\bf 54} (2014), 295--310.
	
	\bibitem{CGSZ18} O.~Celikbas, M.~Gheibi, A.~Sadeghi and M.R.~Zargar, {\it Homological dimensions of rigid modules}, Kyoto J. Math. {\bf 58} (2018), 639--669.
	
	
	\bibitem{CS16} Olgur Celikbas and Sean Sather-Wagstaff, {\it Testing for the Gorenstein property}, Collect. Math. {\bf 67} (2016), 555--568.
	
	\bibitem{Cr00} L.W.~Christensen, {\it Gorenstein dimensions}, Lecture Notes in Mathematics {\bf 1747}, Springer-Verlag, Berlin, 2000.

%
	\bibitem{CHKV06} A. Corso, C. Huneke, D. Katz and W.V. Vasconcelos, {\it Integral closure of ideals and annihilators of homology} in Commutative Algebra, Lect. Notes Pure Appl. Math. 244, Chapman \& Hall/CRC, Boca Raton, FL, 2006, 33--48.
	
	
	\bibitem{DG22} S.~Dey and D.~Ghosh, {\it Complexity and rigidity of Ulrich modules, and some applications}, to appear in Math. Scand., \href{https://arxiv.org/pdf/2201.00984.pdf}{arXiv:2201.00984}.
	
%
%
	
	\bibitem{GP90} V.N.~Gasharov and I.~Peeva, {\it Boundedness versus periodicity over commutative local rings}, Trans. Amer. Math. Soc. {\bf 320} (1990), 569--580. 
	
	\bibitem{GGP18} D.~Ghosh, A.~Gupta and T.J.~Puthenpurakal, {\it Characterizations of regular local rings via syzygy modules of the residue field}, J. Commut. Algebra {\bf 10} (2018), 327--337.
	
	\bibitem{Go87} S. Goto, {\it Integral closedness of complete-intersection ideals}, J. Algebra {\bf 108} (1987), 151--160. 
	
%
%
%
%
%
%
%
%
%
%
%
%
%

%
	
	\bibitem{Mat86} H. Matsumura, {\it Commutative Ring Theory}, Cambridge University Press, Cambridge, 1986.
	
	
	\bibitem{PS73} C.~Peskine and L.~Szpiro, {\it Dimension projective finie et cohomologie locale. Applications \`{a} la d\'{e}monstration de conjectures de M.~Auslander, H.~Bass et A.~Grothendieck}, Inst. Hautes \'{E}tudes Sci. Publ. Math., No. 42, (1973), 47--119.
	
	\bibitem{Rob87} P. Roberts, {\it Le th\'{e}or\`{e}me d'intersection}, C.R. Acad. Sci. Paris 304 (1987), 177--180.
	
	
	\bibitem{SH06} I. Swanson and C. Huneke, {\it Integral closure of ideals, rings, and modules}, London Mathematical Society Lecture Note Series {\bf 336}, Cambridge University Press, Cambridge, 2006.
	
%
%
	
	\bibitem{Ta19} E.~Tavanfar, {\it Test modules, weakly regular homomorphisms and complete intersection dimension}, to appear in J. Commut. Algebra, \href{https://arxiv.org/pdf/1911.11290.pdf}{arXiv:1911.11290}.
	

	
\end{thebibliography}
\end{document}